\documentclass[11pt,a4paper,reqno]{amsart}
\usepackage{amsmath,amscd,amstext,amsthm,amsfonts,latexsym}
\usepackage[dvips]{graphicx}
\usepackage{enumitem}
\usepackage{bbm}
\usepackage[colorlinks=true, linkcolor=blue, urlcolor=red, citecolor=blue, hyperindex, backref]{hyperref}

\theoremstyle{plain}
\newtheorem{theorem}{Theorem}[section]
\newtheorem{proposition}[theorem]{Proposition}
\newtheorem{lemma}[theorem]{Lemma}
\newtheorem{corollary}[theorem]{Corollary}
\newtheorem{maintheorem}{Theorem}

\newtheorem{maincorollary}{Corollary}

\theoremstyle{definition}
\newtheorem{remark}[theorem]{Remark}
\newtheorem{example}[theorem]{Example}
\newtheorem{definition}[theorem]{Definition}

\newtheorem{question}{Question}

\newcommand{\field}[1]{\mathbb{#1}}
\newcommand{\RR}{\field{R}}
\newcommand{\CC}{\field{C}}

\newcommand{\NN}{\field{N}}

    \newcommand{\Ga}{\Gamma}
       
\newcommand{\ep} {\varepsilon}

\newcommand{\la} {\lambda}

\newcommand{\si} {\sigma}

\newcommand{\cU}{{\mathcal U}}

\newcommand{\cL}{{\mathcal L}}
\newcommand{\cP}{{\mathcal P}}

\newcommand{\cD}{{\mathcal D}}

\newcommand{\cE}{{\mathcal E}}

\newcommand{\cC}{{\mathcal C}}

\newcommand{\cH}{{\mathcal H}}

\newcommand{\diam}{\operatorname{{diam}}}

\newcommand{\graph}{\operatorname{{graph}}}

\newcommand{\Leb}{\operatorname{{Leb}}}

\newcommand{\supp}{\operatorname{{supp}}}

\title{Metric mean dimension, H\"older regularity and Assouad spectrum}

\begin{document}

\author[A. Baraviera]{Alexandre Tavares Baraviera}
\address{Instituto de Matemática e Estatística - UFRGS, Av. Bento Gonçalves, 9500 - Agronomia, Porto Alegre - RS, 91509-900, Brasil}
\email{atbaraviera@gmail.com}

\author[M. Carvalho]{Maria Carvalho}
\address{CMUP \& Departamento de Matem\'atica, Faculdade de Ci\^encias da Universidade do Porto, Rua do Campo Alegre 687, Porto, Portugal.}
\email{mpcarval@fc.up.pt}

\author[G. Pessil]{Gustavo Pessil}
\address{CMUP \& Departamento de Matem\'atica, Faculdade de Ci\^encias da Universidade do Porto, Rua do Campo Alegre 687, Porto, Portugal.}
\email{gustavo.pessil@outlook.com}




\date{\today}

\begin{abstract}
Metric mean dimension is a geometric invariant of dynamical systems with infinite topological entropy. We relate this concept with the fractal structure of the phase space and the H\"older regularity of the map. Afterwards we improve our general estimates in a family of interval maps by computing the metric mean dimension in a way similar to the Misiurewicz formula for the entropy, which in particular shows that our bounds are sharp. As an application, we determine the metric mean dimension of the classical Weierstrass functions. Of independent interest, we develop a dynamical analogue of the Minkowski-Bouligand dimension for subshifts on Ahlfors regular alphabets, which also provides an entropy formula in terms of the size of the set of admissible words, generalizing the classical result for subshifts on finite alphabets.
\end{abstract}

\maketitle

\tiny
\normalsize

\section{Introduction}

The metric mean dimension was introduced by  E. Lindenstrauss and B. Weiss in \cite{LW2000} to quantify the complexity of  dynamical systems with infinite topological entropy. It vanishes if the entropy is finite; otherwise, it measures the speed at which the entropy at a given scale approaches infinity as this scale
goes to zero. In this sense, the route from entropy to metric mean dimension is a dynamical analogue of the one from counting to box dimension. In both cases, the choice of the metric has impact precisely on the speed of convergence, and therefore the metric mean dimension is not purely topological, although it is invariant under bi-Lipschitz conjugacies. In this article, we investigate the interplay between the metric mean dimension, the geometry of the underlying space and the regularity of the map.

\smallskip

Lipschitz maps on compact metric spaces of finite dimension have finite topological entropy (see \eqref{entropyineq}), hence their metric mean dimension is zero. In \cite{Hazard}, P. Hazard presented a family of H\"older continuous interval maps with infinite entropy. Thus, it is worthwhile studying the metric mean dimension within this family, and we wonder if and how it is related to the H\"older exponents of the maps. We answer this question positively for a $2-$parameter family of interval maps, which comprises Hazard's examples (see Subsection~\ref{app} and Theorem~\ref{mdiminterval}). The core of the proof of this result is a formula for the metric mean dimension in terms of a sequence of horseshoes and the scales at which their dynamics can be detected, resembling Misiurewicz formula for the entropy (cf. \cite{Mis}). We deduce this formula by developing a dynamical analogue of the Minkowski-Bouligand dimension on Ahlfors regular spaces (see Subsection~\ref{ssecentropyformula} and Theorem~\ref{mdimsubshift}).

\smallskip

Taking into account that the metric mean dimension has both a dynamical and a geometrical imprint, when the phase space is not as homogeneous as the interval we need to deal with more refined geometric structures, which may not be detected if one chooses an unfit notion of dimension. It turns out that the more recent concept of Assouad spectrum (cf. \cite{FY2018}) captures in a precise way the impact of the H\"older regularity on the metric mean dimension (see Subsection~\ref{fractaldimensions} and Theorem~\ref{holderineq}). The adequacy of the Assouad spectrum in this setting is due to the fact that, at a given parameter $\alpha\in(0,1)$, it essentially quantifies how many $\ep$-balls are necessary to cover an $\ep^\alpha$-ball. Thus, it is naturally associated to dynamical covers generated by an $\alpha$-H\"older map. The essence of our approach is the application of \cite[Theorem B and Lemma 6.1]{P2024} which provide bounds for the metric mean dimension in terms of the spectral radii of matrices whose entries are ruled by the H\"older regularity of the map.

\smallskip

In the next subsections, we will describe in detail our mathematical contribution and discuss some examples of interest. After a brief glossary with the main definitions and preliminary information, we present the proofs. We end the paper raising a few open questions suggested by this work.

\subsection{H\"older regularity}

Let $(X,d)$ be a compact metric space and $T\colon X\to X$ be a continuous map. We say that $T$ is \emph{Lipschitz} if there is some $C>0$ such that $$ d(T(x),T(y))\,\leq\,C\,d(x,y)\qquad\forall\,x,y\,\in\,X. $$
It is known that the topological entropy of a Lipschitz map $T$ has a natural upper bound in terms of the Lipschitz constant and the dimension of the phase space (cf. \cite[Theorem~3.2.9]{KH}), namely\begin{equation}\label{entropyineq}
    h_{top}(T)\,\leq\,\max\,\{0,\log C\}\,\overline{\dim}_B(X,d).
\end{equation}
However, if one is interested in maps satisfying weaker forms of regularity, the phenomenon of infinite entropy may occur. Given $\alpha\in(0,1)$, we say that $T$ is \emph{$\alpha-$H\"older} if there exists a constant $C_\alpha>0$ such that$$ d(T(x),T(y))\,\leq\,C_\alpha\,d(x,y)^\alpha\qquad\forall\,x,y\,\in\,X. $$  We sometimes refer to $\alpha$ as the \emph{H\"older exponent of $T$.} Hazard (cf. \cite{Hazard}) presented a family of interval endomorphisms with infinite topological entropy which are H\"older continuous. This indicates that the metric mean dimension can be a useful tool to study such systems. We refer the reader to Subsection~\ref{app}, where we discuss this type of maps.
  
The following result is an analogue of \eqref{entropyineq} for the upper metric mean dimension $\overline{\mathrm{mdim}}_M(X,d,T)$ on spaces with finite Assouad dimension $\mathrm{dim}_A(X,d)$ in terms of the H\"older exponent $\alpha$ and the Assouad spectrum $\mathrm{dim}_A^\alpha(X,d)$ of the phase space at the parameter $\alpha$. See the precise definitions in Section~\ref{prelim}.

\begin{maintheorem}\label{holderineq}
    Let $(X,d)$ be a compact metric space with finite Assouad dimension. If $\alpha\,\in\,(0,1)$ and $T\colon X\to X$ is $\alpha-$H\"older, then 
            \begin{equation*}
        \overline{\mathrm{mdim}}_M(X,d,T) \, \leq \, (1-\alpha) \dim_A^{\alpha}(X,d).
        \end{equation*}  If, in addition, $\dim_A(X,d)\,>\,0$ then
            \begin{equation*}
                \sup\,\{\alpha\in(0,1)\,:\, T\text{ is  }\alpha-\text{H\"older}\} \, \leq\, 1\,-\,\frac{\overline{\mathrm{mdim}}_M(X,d,T)}{\dim_A(X,d)}.
            \end{equation*}  
\end{maintheorem}

The previous inequalities are sharp: in Subsection~\ref{app} we present a family of examples for which the equalities are attained. Theorem~\ref{holderineq} answers the questions raised in \cite{ARA}.

\smallskip

Recall that, if a compact metric space admits a Lipschitz map with infinite entropy, then by \eqref{entropyineq} its box dimension is necessarily infinite. Taking into account the connection between the box dimension and the Assouad spectrum (see \eqref{ineqdimensions}), we deduce from Theorem~\ref{holderineq} that the existence of a H\"older endomorphism with full metric mean dimension also imposes constraints on the fractal structure of the underlying space.

\begin{maincorollary}
     Let $(X,d)$ be a compact metric space with finite Assouad dimension. If there exists a continuous map $T\colon X\to X$ such that $\overline{\mathrm{mdim}}_M(X,d,T)=\overline{\dim}_B(X,d)$, then $$\dim_A^\alpha(X,d)\,=\,\frac{\overline{\dim}_B(X,d)}{1-\alpha}\quad\quad\forall\,\alpha\,\in\,\cH(T),$$ where $\cH(T)\,=\, \{\alpha\in(0,1)\,:\, T\text{ is  }\alpha-\text{H\"older}\}\,\cup\,\{0\}\,.$
\end{maincorollary}
    
In particular, if the space is sufficiently homogeneous so that the upper box and Assouad dimensions coincide, then it does not admit a H\"older endomorphism with full upper metric mean dimension. This is the case when the space is Ahlfors regular (see Definition~\ref{dimhommms}).

\begin{maincorollary}\label{notholder}
     Let $(X,d)$ be a compact metric space such that $$0<\overline{\dim}_B(X,d)=\dim_A(X,d)<+\infty.$$ If a continuous map $T\colon X\to X$ satisfies $\overline{\mathrm{mdim}}_M(X,d,T)=\overline{\dim}_B(X,d)$, then it is not $\alpha-$H\"older for any $\alpha\,\in\,(0,1).$
\end{maincorollary}

Since a $C^0-$Baire generic homeomorphism on a compact manifold of dimension greater than $1$ has full metric mean dimension (cf. \cite{CRV2}), the previous corollary provides an alternative proof of the well-known fact that a $C^0-$generic homeomorphism is not $\alpha-$H\"older for any $\alpha\in(0,1)$. A similar conclusion holds for a $C^0-$generic endomorphism of the interval.

\smallskip

Here are two concrete examples on which we may apply Theorem~\ref{holderineq}. Consider the interval $[0,1]$ with the Euclidean distance, which we denote by $|\cdot|$.

\begin{example}[Weierstrass functions]
Fix positive real numbers $b>1$ and $1/b\,<\,a\,<\,1$, and consider the map $W_{a,b}\colon [-c,c]\,\to\,[-c,c]$, where $c=(1\,-\,a)^{-1}$, defined by 
$$W_{a,b}(x)\,=\,\sum_{n=0}^\infty a^{n}\,\cos( 2\pi b^n x).$$ 
It is known that $W_{a,b}$ is $\beta-$H\"older, where $\beta\,=\,\frac{\log (1/a)}{\log b}$ (cf. \cite{Har}). On the other hand, $W_{a,b}$ is $\beta-$hypersensitive, that is, $\exists \delta,C>0$ such that, for every interval $J\subset [-c,c]$ with diameter smaller than $\delta,$ we have $\diam\big(W_{a,b}(J)\big)\,\geq\, C\, \diam(J)^\beta$ (cf. \cite{FracStocV}). Combining this information with Theorem~\ref{holderineq} and \cite[Theorem 4.1]{KM2022}, we conclude that $$\textrm{mdim}_M([0,1],|\cdot|,W_{a,b})\,\,=\,\,1\,-\,\beta \,\,=\,\, 1\,+\,\frac{\log a}{\log b}.$$ We observe that the Hausdorff/box dimension of the graph of $W_{a,b}$ is $2\,+\,\log a/\log b$ (cf. \cite{graphW})
\end{example}

\begin{example}[Hypersensitive zipper maps]
Let us consider the parameterized family of maps defined in \cite[Section 1.1]{KM2022}. Denote by $\cC_0^0$ the space of continuous functions $f\colon [0,1]\to[0,1]$ fixing $0$ and $1$, with the supremum norm. Consider a pair of points $p=\big( (x_1,y_1),(x_2,y_2)\big)\in(0,1)^4$ such that $x_2>x_1$ and $y_2<y_1$. Define the operator $\Phi_p\colon \cC_0^0\to\cC_0^0$ by $$\Phi_p f(x)\,\,=\,\, \left\{\begin{array}{lc}
    y_1\,\, f(\frac{x}{x_1}) & \text{if }\,\,x\in[0,x_1] \\    
    y_1 \,-\,(y_1-y_2)f\big(\frac{x-x_1}{x_2-x_1}\big)  & \,\,\text{if }\,\,x\in [x_1,x_2]\\
     y_2 \,+\,(1-y_2)f\big(\frac{x-x_2}{1-x_2}\big) & \text{if }\,\,x\in[x_2,1]. 
\end{array}
\right.$$
Then $\Phi_p$ is a contraction and thus has a unique fixed point $Z_p\colon [0,1]\to[0,1]$, which is called the zipper map of parameter $p$. The point $p$ brings forward $4$ additional parameters, denoted by $\la_{\min}(p)\leq\la_{\max}(p)$ and $0<h_{min}(p)\leq h_{max}(p)<1$, which describe geometric properties of the operator and impart dynamical information about $Z_p$. For instance, in \cite[Proposition~H]{KM2022} it was shown that $Z_p$ is $\beta-$hypersensitive if and only if $\la_{min}>1$, where $\beta= 1 - \frac{\log \la_{min}}{|\log h_{min}|}$. Consequently (see \cite[Theorem~D]{KM2022}), 
 $$\frac{\log \la_{min}}{|\log h_{min}|}\,\,\leq\,\,\underline{\textrm{mdim}}_M([0,1],|\cdot|,Z_p).$$ Besides, $Z_p$ is $\alpha-$H\"older for $\alpha=1\,-\,\frac{\log \la_{max}}{|\log h_{max}|}$ (cf. \cite[Proposition 2.2]{KM2022}). Combining this information with Theorem~\ref{holderineq}, we obtain $$\frac{\log \la_{min}}{|\log h_{min}|}\,\,\leq\,\,\underline{\textrm{mdim}}_M([0,1],|\cdot|,Z_p)\,\,\leq\,\,\overline{\textrm{mdim}}_M([0,1],|\cdot|,Z_p)\,\,\leq\,\,\frac{\log \la_{max}}{|\log h_{max}|}.$$
\end{example}

\medskip

\subsection{Horseshoes in the interval}\label{app}


In this subsection we consider the interval $[0,1]$ with the Euclidean distance and the Lebesgue measure, which we denote by $\Leb$. We will compute the metric mean dimension of a $2-$parameter family of interval maps firstly introduced by Kolyada and Snoha (cf. \cite{KS}). This formula refines the information provided by Theorem~\ref{holderineq} for this family, revealing that any loss in regularity is translated into metric mean dimension and vice-versa.

 Fix a sequence $a\,=\,(a_k)_{k\,\geq\,0}$ of real numbers and a sequence $b\,=\,(b_k)_{k\,\geq\,1}$ of positive odd integers satisfying the following conditions:

 \smallskip

\begin{itemize}
    \item[(C1)] $0=a_0<a_1<...<a_k\to1 $.
    \smallskip
    \item[(C2)]  $(a_{k}-a_{k-1})_k$ decreases to zero.
    \smallskip
    \item[(C3)] $(b_k)_k$ is strictly increasing.
\end{itemize}

\smallskip

\noindent For each $k \in \mathbb N$, denote by $f_k\colon[0,1]\to[0,1]$ the unique continuous piecewise affine map with $b_k$ equidistributed full branches such that $f_k$ is increasing in the first one - and since $b_k$ is odd, also in the last. Let $J_k=[a_{k-1},a_k]$ and $T_k\colon J_k \to [0,1]$ be the unique increasing affine map from $J_k$ onto $[0,1]$.  Define

\begin{equation}\label{Tab}
    \begin{array}{rccc}
T_{a,b} \colon & [0,1] & \rightarrow & [0,1] \\
& x \in J_k & \mapsto & T_k^{-1}\circ f_k \circ T_k\\
& x=1 & \mapsto & 1
\end{array}
\end{equation}

\begin{figure}[!htb]
\begin{center}
\includegraphics[scale=0.4]{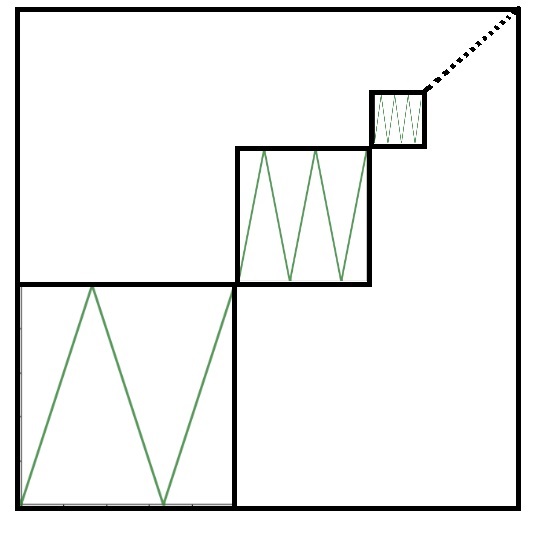}
\caption{Graph of the map $T_{a,b}$.}
\label{figure1}
\end{center}
\end{figure}

Each map is determined by a sequence of horseshoes accumulating at the extreme point $1$ (see Figure~\ref{figure1}). The sequence $(a_k)_k$ defines the intervals $(J_k)_k$ where the horseshoes lie, and each $b_k$ represents the number of branches in the $k-$th horseshoe. In particular, the topological entropy of $T_{a,b}$ restricted to $J_k$ is $\log b_k$. To fix notation we write $\ep_k = |J_k|/b_k$, which represents the length of the injectivity domains in $J_k$.

Clearly, $T_{a,b}$ has infinite topological entropy for every $a$ and $b$ (see \cite{KS}). The metric mean dimension has been computed for specific choices of the parameters (see for instance \cite{VV} and \cite{CPV}). Similar maps were studied by Hazard in \cite{Hazard}, though without any mention to metric mean dimension. Actually, the main focus there was to construct interval maps which are $\alpha-$H\"older, for $\alpha\in(0,1),$ and whose topological entropy is infinite.

A classical result of Misiurewicz \cite{Mis} states that if $T\colon[0,1]\to[0,1]$ is a continuous map then for every $n\in\NN$ there exists \begin{itemize}
    \item an interval $J_n$
    \item a collection $\cD_n$ of pairwise disjoint subintervals of $J_n$
    \item $k_n\in\NN$
\end{itemize} such that $J_n\subset T^{k_n}(I)$ for every $I\in\cD_n$ and \begin{equation}\label{miseq}
    h_{top}(T)\,=\,\limsup_{n\,\to\,+\infty}\frac{\log\#\cD_n}{k_n}.
\end{equation} 
Observe that Misiurewicz formula \eqref{miseq} for the maps $T_{a,b}$ defined in \eqref{Tab} becomes $$ h_{top}(T_{a,b})\,=\,\limsup_{k\,\to\,+\infty}\,\log b_k\,=\,+\infty.$$
We will obtain a formula similar to \eqref{miseq} for the metric mean dimension, with respect to the Euclidean distance, within the family $T_{a,b}.$ As a consequence, we also establish a connection between the upper metric mean dimension and the H\"older regularity of $T_{a,b}$, which shows that the estimates in Theorem~\ref{holderineq} are sharp.

\begin{maintheorem}\label{mdiminterval}
   Let $a\,=\,(a_k)_{k\,\geq\,0}$ and $b\,=\,(b_k)_{k\,\geq\,1}$ be two sequences satisfying the conditions \emph{(C1)-(C3)} and $T_{a,b}$ be the corresponding interval map. Then  
   \begin{eqnarray}\label{mismdim}       \nonumber\overline{\mathrm{mdim}}_M([0,1],|\cdot|,T_{a,b})&=&\limsup_{k\,\to\,+\infty}\,\frac{\log b_k}{\log(1/\ep_k)} \\ \nonumber\underline{\mathrm{mdim}}_M([0,1],|\cdot|,T_{a,b})&=&\liminf_{k\,\to\,+\infty}\,\frac{\log b_k}{\log(1/\ep_k)}.
   \end{eqnarray}
Moreover, 
    \begin{equation}\label{11111}
        \sup\,\cH(T)\,=\,1\,-\,\overline{\mathrm{mdim}}_M([0,1],|\cdot|,T_{a,b}),
    \end{equation}

\medskip
    
    \noindent where $\cH(T)\,=\,\{\alpha\in(0,1)\,:\, T\text{ is  }\alpha-\text{H\"older}\}\,\cup\,\{0\}.$
\end{maintheorem}

\medskip


For example, by Theorem~\ref{mdiminterval}:
\begin{itemize}
    \item[(i)] If $a_k\,=\,\sum_{n=1}^k 6/\pi n^2$ and $b_k\,=\,3^k$, then $\mathrm{mdim}_M([0,1],|\cdot|,T_{a,b})\,=\,1$ and $T_{a,b}$ is not $\alpha-$H\"older for any $\alpha\in(0,1)$.
     \item[(ii)] Given $0<\beta<1$, if $a_k\,=\, \sum_{n=1}^k C(\beta)\,3^{n\,(1-\frac{1}{\beta})}$, where $C(\beta)\,=\,\big(\sum_{n=1}^\infty 3^{n\,(1-\frac{1}{\beta})}\big)^{-1}$, and $b_k\,=\,3^k$, then $\mathrm{mdim}_M([0,1],|\cdot|,T_{a,b})\,=\,\beta$ and $T_{a,b}$ is $\alpha-$H\"older for every $\alpha\in(0,1-\beta)$.
    \item[(iii)]  If $a_k\,=\, 1/2^{k}$ and $b_k\,=\,2k+1$, then $\mathrm{mdim}_M([0,1],|\cdot|,T_{a,b})\,=\,0$ and $T_{a,b}$ is $\alpha-$H\"older for every $\alpha\in(0,1).$
\end{itemize}

\smallskip

We stress that the lack of regularity represented by the value of the metric mean dimension in Theorem~\ref{mdiminterval} is of a different nature from the one studied by Hazard in \cite{Hazard}. The maps in \cite{Hazard} are defined like we did for $T_{a,b}$, considering $a=(2^{-k})_k$ and $b=(2k+1)_k$, but replacing each $f_k$ by $g_k = \phi_\alpha\circ f_k \circ \phi_\alpha^{-1}$, where $\phi_\alpha(x)\,=\,x^\alpha$. Thus, these maps have zero metric mean dimension, and their irregularity is due to the appearance of cusps. On the other hand, we consider maps composed by sequences of piecewise \emph{linear} horseshoes whose lack of regularity comes from the distortion ratio between the lengths of the injectivity domains of the horseshoes and of their images; or equivalently, the ratio between the number of branches and the scales at which they separate. This is precisely what metric mean dimension quantifies.

\medskip

\subsection{An entropy formula}\label{ssecentropyformula}

It is known \cite[Theorem 7.13]{Wa} that the topological entropy of a subshift on finitely many symbols $Y\subset\{1,\cdots,q\}^\NN$ can be expressed in terms of the exponential growth of the number of $n-$admissible words in $Y$, namely, \begin{equation}\label{3entshift}
    h_{top}(Y,\si_\NN)\,\,=\,\,\lim_{n\,\to\,+\infty}\frac{1}{n}\log \Theta_n(Y), 
\end{equation} 
where, for each $n\,\in\,\NN$, $\Theta_n(Y)$ is defined as the number of $n-$tuples $[i_1,...,i_n]$ such that the set $\{(y_k)_{k\,\in\,\NN}\,\in\,Y\,:\, y_1\,=\,i_1,\cdots,\,y_n\,=\,i_n\}$ is nonempty. We will extend this formula for subshifts on more general alphabets.

A compact metric space $(X,d)$ is \emph{Ahlfors regular} if there exists a Borel probability measure $\mu\,\in\,\cP(X)$ satisfying \begin{equation*}
             \exists \,\ep_0,C_1,C_2>0: \qquad  C_1\, \ep^{\dim_B X}\,\leq\,\mu(B_\ep(x))\,\leq \,C_2\, \ep^{\dim_B X}\qquad \forall x\in X,\,\forall0<\ep<\ep_0.
\end{equation*}
For instance, a finite set with the normalized counting measure or the Euclidean cube with the Lebesgue measure are Ahlfors regular. In this setting, we obtain sharp estimates for the $\ep-$entropy of an arbitrary subshift $Y\subset X^\NN$ on an Ahlfors regular alphabet $X$, namely
\begin{equation*}
    h_\ep(Y,d_\NN,\si_\NN)\,\,\,=\,\,\, \lim_{n\,\to\,+\infty}\frac{1}{n}\log \Big(\frac{\mu^n\big(Y(n,\ep)\big)}{\ep^{n\,\dim_B X}}\Big)\,\,\,+\,\,\,o(1),
\end{equation*}
where $Y(n,\ep)$ is the set $B_\ep(\pi_n(Y))\subset X^n$ with respect to the supremum norm $||\cdot||_\infty$ and $\pi_n\colon X^\NN\,\to\,X^n$ is the projection on the first $n$ coordinates.  

\medskip

\begin{maintheorem}\label{mdimsubshift}
    Let $(X,d,\mu)$ be an Ahlfors regular space. For every subshift $Y\,\subset\,X^\NN$, there exists $\la=\la(Y)>0$, which is unique if $h_{top}(Y,\si_\NN)<+\infty$, such that
    \begin{equation}\label{entformula}
         h_{top}(Y,\si_\NN)\,\,\,=\,\,\, \limsup_{\ep\,\to\,0^+}\lim_{n\,\to\,+\infty}\frac{1}{n}\log \Big(\frac{\mu^n\big(Y(n,\ep)\big)}{\la^n\,\ep^{n\,\dim X}}\Big),
    \end{equation}
In addition, independently of $\la$,
    
    \begin{equation}\label{eq1mdimsubshift}
            \begin{split}
            \overline{\mathrm{mdim}}_M(Y,d^\NN ,\si_\NN)\, \,\,\, & =\,\,\,\,  \mathrm{dim}(X,d)\,+\, \limsup_{\ep\,\to\,0^+} \frac{1}{\log(1/\ep)} \lim_{n\,\to\,+\infty}\frac{1}{n}\log \mu^n \big( Y(n,\ep) \big)  \end{split} 
    \end{equation} \begin{equation*}\label{eq2mdimsubshift}
            \begin{split}                \underline{\mathrm{mdim}}_M(Y,d^\NN ,\si_\NN)\, \,\,\, & =\,\,\,\,  \mathrm{dim}(X,d)\,+\, \liminf_{\ep\,\to\,0^+} \frac{1}{\log(1/\ep)} \lim_{n\,\to\,+\infty}\frac{1}{n}\log \mu^n \big( Y(n,\ep) \big).
            \end{split} 
    \end{equation*}

\end{maintheorem}

\medskip

In particular, if $X$ is finite we recover \eqref{3entshift}. Indeed, if $X\,=\,\{1,\cdots,q\}$ then it has zero dimension and is an Ahlfors regular space with respect to the normalized counting measure $\mu$, whose constants are $C_1\,=\,C_2\,=\,1/q$. The proof of Theorem~\ref{mdimsubshift} yields $\la(Y)=1/q$ for every $Y$ and $$ h_{top}(Y,\si_\NN)\,\,=\,\, 
\limsup_{\ep\,\to\,0^+} \lim_{n\to\,+\infty}\frac{1}{n}\log \Big(q^n\cdot\mu^n\big(Y(n,\ep)\big) \Big) \,\,=\,\,\lim_{n\,\to\,+\infty}\frac{1}{n}\log \Theta_n(Y).$$

It is known (cf. \cite[Theorem B]{P2024}) that one can compute the metric mean dimension of a continuous map $T$ acting on a compact metric space $X$ by considering the graph of the map as a transition set and the corresponding subshift of compact type
$$Y\,\,=\,\,\big\{ (x,T(x),T^2(x),\cdots )\,:\, x\,\in\,X \big\}.$$
As an application of Theorem~\ref{mdimsubshift}, we obtain an alternative formula for the metric mean dimension of $T$ whenever $X$ is an Ahlfors regular space. This has the advantage of providing an interpretation of the difference between the metric mean dimension of the map and the dimension of the underlying space.

\begin{corollary}\label{lowdimformula}
     Let $(X,d,\mu)$ be an Ahlfors regular space and $T\colon X\to X$ be a continuous map. Then, there exists $\la=\la(T)>0$, which is unique if $h_{top}(X,T)<+\infty$, such that
    \begin{equation*}
         h_{top}(X,T)\,\,\,=\,\,\, \limsup_{\ep\,\to\,0^+}\lim_{n\,\to\,+\infty}\frac{1}{n}\log \Big(\frac{\mu^n\big(Y(T,n,\ep)\big)}{\la^n\,\ep^{n\,\dim X}}\Big).
    \end{equation*}In addition, independently of $\la$,
    \begin{equation*}\label{defup}
            \begin{split}
               \overline{\mathrm{mdim}}_M(X,d,T)\,  & =\,  \mathrm{dim}(X,d)\,+\, \limsup_{\ep\,\to\,0^+} \frac{1}{\log(1/\ep)} \lim_{n\,\to\,+\infty}\frac{1}{n}\log \mu^n \big( Y(T,n,\ep) \big)
            \end{split} 
    \end{equation*} \begin{equation*}\label{deflow}
            \begin{split}
              \underline{\mathrm{mdim}}_M(X,d,T)\,  & =\,  \mathrm{dim}(X,d)\,+\, \liminf_{\ep\,\to\,0^+} \frac{1}{\log(1/\ep)} \lim_{n\,\to\, +\infty}\frac{1}{n}\log \mu^n \big( Y(T,n,\ep) \big)
            \end{split} 
    \end{equation*}
    where $$Y(T,n,\ep)\,=\,\bigcup_{x\,\in\, X}  B_\ep(x)\times B_\ep(T(x))\times\cdots\times B_\ep( T^{n-1}(x)).$$
\end{corollary}

The previous corollary turns out to be a useful tool to explicitly compute the metric mean dimension on Ahlfors regular spaces. See, for instance, Proposition~\ref{mdiminterval1}.

It is well known that the  metric mean dimension of iterates of a continuous map behaves differently from the topological entropy. In fact, we have $$\overline{\mathrm{mdim}}_M(X,d ,T^k)\,\,\,\leq\,\,\, k\,\,\,\overline{\mathrm{mdim}}_M(X,d ,T),\quad \forall k\geq1,$$ whereas this inequality can be strict, as happens when the equality in \eqref{mdimleqbox} is attained. As a consequence of Corollary~\ref{lowdimformula}, we obtain an optimal converse in this setting.

\begin{corollary}\label{iterados}
     Let $(X,d,\mu)$ be an Ahlfors regular space and $T\colon X\to X$ be a continuous map. Then, for every $k\geq1$,
     \begin{eqnarray*}    \overline{\mathrm{mdim}}_M(X,d,T^k)\,&\geq&\,\,\,k\,\,\,\overline{\mathrm{mdim}}_M(X,d,T)\,\, -\,\, (k-1)\,\dim X \\ \underline{\mathrm{mdim}}_M(X,d,T^k)\,&\geq&\,\,\,k\,\,\,\underline{\mathrm{mdim}}_M(X,d,T)\,\, - \,\,(k-1)\,\dim X.
     \end{eqnarray*}

     \noindent In particular, if $T$ has full metric mean dimension, so do all of its iterates.
\end{corollary}

\medskip

\section{Preliminaries}\label{prelim}

In this section we introduce some basic terminology.

\subsection{Metric mean dimension}

Let $(X,d)$ be a compact metric space and $T\colon X \to X$ be a continuous map. For each $n \in \mathbb{N}$, define the Bowen metric
$$d_n(x,y) \,=\, \max_{0\,\leq\, j\,\leq\, n-1}\,d(T^j(x),\,T^j(y)) \qquad \forall\, x, y \in X$$
which is equivalent to $d$. We sometimes refer to balls with respect to this metric as dynamical balls.

\smallskip

Given $\ep > 0$ and $K\subset X$, consider the following minimum
\begin{equation}\label{def:S}
S(K,d,\ep) \,=\, \min \left\{ \ell\colon \,\{U_i\}_{1\,\leq \,i\,\leq \,\ell} \text{ is a finite open cover of $K$ with $\mathrm{diam}(U_i, d) \leq \ep$}\right\}
\end{equation}
and the limit
$$h_\ep(K,d,T) \, = \, \lim_{n\,\to\,+\infty} \frac{1}{n}\log S(K,d_n,\ep),$$
which exists since the sequence $\big(\log S(K,d_n,\ep)\big)_n$ is sub-additive in the variable $n$. Recall that the \emph{topological entropy} is given by $h_{top}(X,T)\,=\,\lim_{\ep\,\to\,0^+}h_\ep(X,d,T).$ 

\begin{definition}\label{mdimT}
\emph{The upper/lower metric mean dimension of $(X,d,T)$ are given, respectively, by
$$\overline{\mathrm{mdim}}_M(X,d,T) \,=\, \limsup_{\ep\,\to\, 0^+}\, \frac{h_\ep(X,d,T)}{\log (1/\ep)}$$ and $$\underline{\mathrm{mdim}}_M(X,d,T) \,=\, \liminf_{\ep\,\to\, 0^+}\, \frac{h_\ep(X,d,T)}{\log (1/\ep)}.$$}
\end{definition}

\smallskip

A subset $E\subset X$ is said to be $\ep-$separated with respect to the metric $d$ if $d(x,y) \geq \ep$ for every $x, y \in E$; it is $\ep-$spanning with respect to the metric $d$ if for every $x\in K$ there exists some $y\in E$ such that $d(x,y)\leq\ep.$ The notion of upper/lower metric mean dimension can be equivalently defined if one replaces $S(X,d,\ep)$ by either
$$\,\,\,S_1(X,d,\ep)\,=\,\sup\Big\{\#E\colon \,E\subset X \text{ is $\ep-$separated}\Big\}$$
or
$$S_2(X,d,\ep) \,=\,\inf\Big\{\#E\colon \,E\subset X \textrm{ is $\ep-$spanning}\Big\}$$ and instead of the limit in $n$, we take either $\limsup_n$ or $\liminf_n.$ 

\medskip
 The \emph{upper/lower box
dimension of $(X,d)$}  are, respectively, defined as
$$\overline{\dim}_B (X,d) \,=\, \limsup_{\varepsilon \, \to \, 0^+}\, \frac{\log S(X,d,\ep)}{\log(1/\varepsilon)}\quad\text{and}\quad\underline{\dim}_B (X,d) \,=\, \liminf_{\varepsilon \, \to \, 0^+}\, \frac{\log S(X,d,\ep)}{\log(1/\varepsilon)}.$$ If the above limits exist, we denote it by $\dim_B(X,d)$ (or simply by $\dim_BX$ if the metric is clear) and refer to it as the \emph{box dimension of $X$}. We recall that (cf. \cite{VV})
\begin{equation}\label{mdimleqbox}
    \begin{split}
        \overline{\mathrm{mdim}}_M(X,d,T) \,&\leq\,\overline{\mathrm{dim}}_B(X,d) \\ \underline{\mathrm{mdim}}_M(X,d,T) \,&\leq\,\underline{\mathrm{dim}}_B(X,d)
    \end{split}.
\end{equation}

\subsection{Other fractal dimensions}\label{fractaldimensions}

In what follows, we denote by $B_\ep(x)$ the open ball on the metric space $(X,d)$ centered at $x$ with radius $\ep.$

\begin{definition}
    \emph{
The Assouad dimension of $(X,d)$ is defined as $$ \dim_A(X,d)\,=\, \inf\,\Bigg\{ s\,:\, \begin{matrix}          \exists\,C,\ep_0>0\,\,\,:\,\,\,\forall\,0<r<R<\ep_0\,\,\, \forall\,x\in X \\ S\big( B_R(x),d,r \big)\,\leq\,C\,\Big(\frac{R}{r}\Big)^s
\end{matrix}    \Bigg\}. $$    
    }
\end{definition}

The Assouad dimension quantifies the maximal local complexity of a fractal, since it provides the maximal exponential growth rate of $\sup_{x\,\in\,X}\,S\big( B_R(x),d,r \big)$ as the parameters $0<r<R$ vary independently. In particular, it is always an upper bound to the box dimension, which is a more averaged measurement of complexity over the whole space. Fraser and Yu proposed in \cite{FY2018} an interpolation between these notions of complexity, obtained by fixing the relation between the scales $r$ and $R$ in the definition of Assouad dimension. More precisely, for each fixed $\theta\in(0,1)$ the scales are chosen to satisfy $R=r^\theta$, yielding a $1-$parameter family of dimensions.

\begin{definition}\label{spectrum}
    \emph{
The Assouad spectrum of $(X,d)$ is the map $$\theta\in(0,1)\,\mapsto\, \dim_A^\theta(X,d)\,=\, \inf\,\Bigg\{ s\,:\, \begin{matrix}          \exists\,C,\ep_0>0\,\,\,:\,\,\,\forall\,0<R<\ep_0\,\,\, \forall\,x\in X \\ S\big( B_R(x),d,R^{1/\theta} \big)\,\leq\,C\,\Big(\frac{R}{R^{1/\theta}}\Big)^s
\end{matrix}    \Bigg\}. $$    
   }
\end{definition}

Since the Assouad spectrum is defined in terms of a unique decreasing scale, it can also be expressed as a limit. In order to maintain a unified notation we write $\ep\,=\,r\,=\,R^{1/\theta}$, obtaining 
\begin{equation}\label{Assouadspectrum}
    \dim_A^\theta(X,d)\,=\,\limsup_{\ep\,\to\,0^+}\,\sup_{x\,\in\,X}\,\frac{\log S\big( B_{\ep^\theta}(x),d,\ep \big)}{(1-\theta)\log (1/\ep)}.
\end{equation}

 The aforementioned notions of dimension are related by the inequalities established in \cite[Proposition 3.1]{FY2018}: for every $\theta\in(0,1),$ \begin{equation}\label{ineqdimensions}
    \overline{\dim}_B(X,d)\,\leq \, \dim_A^\theta(X,d)\,\leq\,\min\bigg\{ \,\frac{\overline{\dim}_B(X,d)}{1-\theta},\,\dim_A(X,d) \bigg\}.
\end{equation}
Moreover, the Assouad spectrum varies continuously in $(0,1)$ and can be extended to $\theta=0$ by $\dim_A^0(X,d)=\overline{\dim}_B(X,d)$. It is also known that, as $\theta\to1^-$, the spectrum converges to the so called \emph{quasi-Assouad dimension}, which is always finite whenever the Assouad dimension is (see \cite{Fraser}). We observe that the Assouad dimension is finite if and only if $(X,d)$ is \emph{doubling}, that is, there exists some $K>0$, which we refer to as the \emph{doubling constant of }$(X,d)$, such that any ball of radius $\ep$ can be covered by $K$ balls of radius $\ep/2$, for every $\ep>0$ (see \cite[Theorem 13.1.1]{Fraser}). For explicit computations and more information on the Assouad spectrum - and its dual notion, the lower spectrum - we refer the reader to the comprehensive references \cite{FY,FY2018}.

\section{Subshifts of compact type}\label{specradiussection}

Let $(X,d)$ be a compact metric space. In what follows we endow the product space $X^\NN$ with the metric $$d^\NN(x,y)\,=\,\sup_{i\,\in\,\NN}\frac{d(x_i,y_i)}{2^{i-1}}.$$ 
A nonempty subset $Y\subset X^\NN$ is called a \emph{subshift} if it is closed and invariant by the shift action  $$\si_\NN\big( (x_i)_{i\,\in\, \NN} \big) = (x_{i+1})_{i\,\in\, \NN}.$$ Consider a subset $\Ga\,\subset X\times X$, whose role will be to prescribe transitions, and the induced set of admissible sequences $$\Ga_\NN\,=\,\{ (x_i)_{i\,\in\, \NN}\,\colon (x_i,x_{i+1})\in\Ga,\,\,\forall\, i\in\NN \},$$ 
which is invariant and we always assume to be non-empty. We refer to $\overline{\Ga_\NN}$ as a \emph{subshift of compact type}, and note that $\Ga_\NN$ is already closed whenever $\Ga$ is. 

Given a continuous map $T\colon X\to X$, we may consider the transition set $$\Ga\,=\,\graph(T)$$ given by its graph. It was shown in \cite{P2024} that the metric mean dimensions of $T$ in $(X,d)$ and of $\si_\NN$ in $(\Ga_\NN,d^\NN)$ coincide:
\begin{equation}\label{invlimit}
\begin{split}
\overline{\mathrm{mdim}}_M(X,d,T)&\,=\,\overline{\mathrm{mdim}}_M(\Ga_\NN,d^\NN,\si_\NN)
\\ \underline{\mathrm{mdim}}_M(X,d,T)&\,=\,\underline{\mathrm{mdim}}_M(\Ga_\NN,d^\NN,\si_\NN)
.\end{split}
\end{equation}
In order to prove Theorem~\ref{holderineq}, we will make use of an upper bound for the $\ep-$entropy of subshifts of compact type in terms of the spectral radius of a suitable matrix, developed in \cite{P2024}. For the sake of completeness, we briefly recall them here.

\subsection{Spectral radius bounds}
 Fix a compact metric space $(X,d)$ and a subset $\Ga\,\subset\,X\times X.$ Given $\ep>0$, let $\cU(\ep)=\{ U_1,...,U_{M(\ep)} \}$ be an open $\ep-$cover of $X.$ Consider the $M(\ep)\times M(\ep)$ matrix $\Ga^\ep$ given by
\begin{equation}\label{gammaepgeneral}
    \Ga_{i,j}^\epsilon = \left \{ \begin{matrix} 1 & \textrm{if } \big( U_i \times U_j \big) \cap \Ga \ne\emptyset \\ 0 & \textrm{otherwise.} \end{matrix} \right. 
\end{equation} 
By counting arguments similar to the ones used for subshifts of finite type, we know that $\|(\Ga^\epsilon)^k\|$  many cylinders composed by elements of $\cU(\ep)$ of size $k$ are enough to cover $\Ga_\NN$, where $\|\cdot\|$ denotes the norm $\|(a_{ij})_{ij}\|=\sum_{i,j}|a_{ij}|$ in the space of $M(\ep)\times M(\ep)$ matrices. This estimate yields the following equivalent formulation of \cite[Lemma 4.1]{Fried}, where $r(A)$ stands for the spectral radius of the matrix $A$.

\begin{lemma}\label{specradgeneral}
   Let $(X,d)$ be a compact metric space and $\Ga\subset X\times X$. Then $$h_\ep(\Ga_\NN,d^\NN,\sigma_\NN) \leq \log r(\Ga^\epsilon).$$
\end{lemma}

We will combine the previous lemma with the following property of matrices to estimate the metric mean dimension.

\begin{theorem}[Gershgorin Circle Theorem, \cite{Gers}]\label{circlethm}
    Let $n\in\NN$ and $A\,=\,(a_{ij})_{i,j}$ be a complex $n \times n$ matrix. Then every eigenvalue of $A$ is contained in the union $$\bigcup_{i=1}^n B^\CC_{R_i}(a_{ii}),$$ where $B^\CC_\delta(z)\,=\,\{w \in\CC: |z-w|\leq\delta\}$ and $R_i\,=\,\sum_{j\ne i}|a_{ij}|$ for every $i=1,...,n.$ 

\end{theorem}

\medskip

 In particular, if we apply Theorem~\ref{circlethm} to a $0-1$ matrix $A$ and its transpose, we get 
\begin{eqnarray}
     &r(A)\,\,&\leq\,\,\,\, \max_{i} \# \{j\,:\,a_{ij}=1\}\label{eqcircle} \\ \label{555}
        r(A)\,=\,&r(A^{T})\,&\leq\,\,\,\,  \max_{j} \# \{i\,:\,a_{ij}=1\}.     
\end{eqnarray}

In the case $X\,=\,[0,1]$ (and also in higher dimensions), the matrix $\Ga^\ep$ can be replaced by another one constructed by using intersections with an $\ep-$grid instead of an open cover. This alternative definition was used by the third named author in \cite[Theorem C]{P2024} to prove that the metric mean dimension of full branched interval maps with infinitely many critical points coincides with the box dimension of the set of such points. To compute this value, the upper bound was obtained by estimating the number of $1$'s in any given \emph{column}, that is, using \eqref{555}. In this paper the estimates will be made for \emph{rows} and we prove Theorem~\ref{holderineq} by using \eqref{eqcircle}.

\medskip

\section{Ahlfors regular spaces}\label{MB}

As a consequence of \cite[Proposition 8.7]{CPV} we know that the local box dimension structure of an alphabet is closely related to the metric mean dimension of the corresponding full shift. In particular, when the space $(X,d)$ is box-dimensionally homogeneous\footnote{We say that $(X,d)$ is \emph{box-dimensionally homogeneous} if $\mathrm{dim}_B(B_\ep(x))\,=\,\mathrm{dim}_B(X),$ for every $x\in X$ and $\ep>0.$}, for every continuous potential $\varphi\colon X^\NN\,\to\,\RR$ we have (cf. \cite[Theorem D]{CPV})$$\overline{\mathrm{mdim}}_M(X^\NN,d^\NN,\si_\NN,\varphi)\,=\, \dim_B(X,d)\,+\,\max_{\mu\,\in\,\cE_{\si_\NN}(X^\NN)}\,\int \varphi\,d\mu,$$ where the upper metric mean dimension with potential was defined in \cite{Tsu2020} and $\cE_{\si_\NN}(X^\NN)$ stands for the set of Borel ergodic $\si_\NN-$invariant probability measures on $X^\NN.$ In what follows, we will approach dimensional homogeneity through a measure-theoretic perspective, which provides a stronger notion of homogeneity, yielding a new formulation for the metric mean dimension. 

For a compact metric space $Z$, denote by  $\cP(Z)$ the space of Borel probability measures on $Z$.

\begin{definition}
    \emph{Given a compact metric space $(X,d)$, we say that $\mu\in\cP(X)$ is Ahlfors regular  if $$\exists \delta\geq0,\,\,\exists,\ep_0,C_1,C_2>0: \qquad  C_1 \,\ep^{\delta}\,\leq\,\mu(B_\ep(x))\,\leq \,C_2 \,\ep^{\delta}\qquad \forall x\in \supp(\mu),\,\forall0<\ep<\ep_0.$$ }
\end{definition}

    Whenever $(X,d)$ admits an Ahlfors regular probability measure $\mu$, we will restrict our study to $\supp(\mu).$ Therefore, without loss of generality, in what follows we will assume $\supp(\mu)=X$. 
    
    The existence of an Ahlfors regular measure on a compact metric space $(X,d)$ reflects the ultimate form of dimensional homogeneity. Namely, it ensures that the Assouad and the lower dimension, which quantify, respectively, the highest and the lowest local complexity in terms of dimension, coincide and are equal to the exponent $\delta$ (see \cite[Theorem 6.4.1]{Fraser}). In particular, all the intermediate notions of dimension such as lower spectrum, packing dimension, Hausdorff dimension, intermediate dimension, box dimension and Assouad spectrum coincide (see \cite{Fa},\cite{FFK},\cite{Fraser} and references therein). Thus, we will always denote $\delta\,=\,\dim(X,d)$, abbreviating into $\dim X$ if the metric is clear. 

    \begin{definition}\label{dimhommms}
         \emph{A triplet $(X,d,\mu)$ is an \emph{Ahlfors regular space} if \begin{equation}\label{ahlforscondition}
             \exists \ep_0,C_1,C_2>0: \qquad  C_1\, \ep^{\dim X}\,\leq\,\mu(B_\ep(x))\,\leq \,C_2\, \ep^{\dim X}\qquad \forall x\in X,\,\forall0<\ep<\ep_0.
         \end{equation} }
    \end{definition}

   The main examples we have in mind are the Euclidean cube $[0,1]^D$ with the Lebesgue measure, a compact Riemannian manifold with the volume measure and (finitely generated) self-similar/self-conformal sets satisfying the open-set-condition with the right dimensional Hausdorff measure (cf. \cite[Corollary 6.4.4]{Fraser}). In fact, the classical notion of Ahlfors regularity is defined with the Hausdorff measure, though it may be formulated for a more general $\mu$ satisfying \eqref{ahlforscondition}. In particular, the only zero dimensional compact Ahlfors regular spaces are the finite ones.

   Every Ahlfors regular space is box-dimensionally homogeneous. However, not every box-dimensionally homogeneous metric space admits an Ahlfors regular measure. An example is the Bedford-McMullen carpet, whose Hausdorff and box dimensions are distinct (cf. \cite{B},\cite{MM}).

\subsection{Minkowski–Bouligand dimension}

The box dimension for compact subsets of the Euclidean space admits an equivalent formulation, the so called Minkowski–Bouligand dimension (see \cite[Proposition 2.4]{Fa}). We briefly recall it here.

 Let $D\in\NN$ and $F\subset\RR^D$ be a compact subset of zero Lebesgue measure (otherwise its box dimension is equal to D) and, for $\ep>0,$ consider its $\ep-$neighbourhood $$B_\ep(F)\,=\,\{ x\in\RR^D: \textrm{dist}(x,F)\,<\,\ep \},$$where $\textrm{dist}(x,F)\,=\,\inf_{a\,\in\,F}\,|x-a|$. This is an open set and it is natural to ask how fast its Lebesgue measure decreases as $\ep\,\to\,0^+$. It turns out that this rate provides information on the dimension, namely, \begin{equation*}\label{boxdim}
    \overline{\mathrm{dim}}_B(F,|\cdot|)\,=\, D + \limsup_{\ep\,\to\,0^+}\frac{\log Leb(B_\ep(F))}{\log(1/\ep)}\quad\text{and}\quad\underline{\mathrm{dim}}_B(F,|\cdot|)\,=\, D + \liminf_{\ep\,\to\,0^+}\frac{\log Leb(B_\ep(F))}{\log(1/\ep)}.
\end{equation*}

We are heading to an analogous result in the dynamical setting. To this end, we need an ambient product space where we can embed and thicken, using $\ep-$neighbourhoods, the space of time $n$ orbits. Afterwards, we evaluate the exponential rate of the $n-$dimensional volume as time increases, and then compute the decay of this quantity as the scale $\ep$ goes to zero. Subshifts of the Euclidean cube are natural candidates for this procedure. In what follows, we consider more general Ahlfors regular alphabets.

\subsection{Proof of Theorem~\ref{mdimsubshift}}

 The main result in this section is an equivalent formulation of the metric mean dimension of subshifts which makes use of the homogeneous structure of the alphabet. The proof is an adaptation of the one in \cite[Proposition 2.4]{Fa}. Let $(X,d,\mu)$ be an Ahlfors regular space. For each $n\geq1$, denote the product measure by $\mu^n\in\cP(X^n).$ Given an arbitrary subshift $Y\subset X^\NN$ and $\ep>0$, let $$Y(n,\ep)\,:=\,  \bigcup_{y\,=\,(y_i)_{i\,\in\,\NN}\,\,\in\,\,Y} B_\ep(y_1)\times B_\ep(y_2)\times\cdots\times B_\ep(y_n)\,\subset\, X^n.$$
We note that $Y(n,\ep)$ is precisely the set $B_\ep(\pi_n(Y))$ with respect to the supremum norm $||\cdot||_\infty$, where $\pi_n\colon X^\NN\,\to\, X^n$ is the projection on the first $n$ coordinates.


Recall that a dynamical ball in the shift space $(X^\NN,d^\NN,\si_\NN)$ is given by $$B_\ep(x,d_n)\,=\, B_\ep(x_1)\times B_\ep(x_2)\times\cdots B_\ep(x_{n})\times B_{2\ep}(x_{n+1})\times\cdots\times B_{2^{\ell}\ep}(x_{n+\ell})\times X\times X\times\cdots$$ where $\ell=\ell(\ep)$ is the smallest positive integer satisfying $2^\ell\ep\geq\diam(X)$.  Observe that, for every $\ep>0$  and $n\geq1$, the measure $\mu^\NN$ satisfies
 \begin{align*}
     \mu^{n+\ell}\big(Y(n+\ell,\ep)\big) \,&=\, \mu^{n+\ell}\bigg(\bigcup_{y\,\in\,Y}B_\ep(y_1)\times\cdots\times B_\ep(y_{n+\ell}) \bigg)\nonumber \\ &\leq\,\mu^\NN \bigg(\bigcup_{y\,\in\,Y}B_\ep(y,d_n)\bigg) \\&\leq\,\mu^n\bigg(\bigcup_{y\,\in\,Y}B_\ep(y_1)\times\cdots\times B_\ep(y_n)\bigg)\,= \,\mu^{n}\big(Y(n,\ep)\big).\nonumber
 \end{align*}
In particular, for the proof we can replace $\mu^n \big(Y(n,\ep)\big)$ by $\mu^\NN \big(Y_{n,\ep}\big)$, where $$Y_{n,\ep}\,=\,\big\{x\,\in\,X^\NN: d_n(x,Y)\leq\ep \big\}.$$

\smallskip

\begin{lemma}\label{subadditive}
    Let $(X,d,\mu)$ be an Ahlfors regular space and $Y\subset X^\NN$ be a subshift. For every $\ep>0$, the sequence $\Big(\log \mu^n\big(Y(n,\ep)\big)\Big)_n$ is subadditive. In particular, the following limits exist $$\lim_{n\,\to\,+\infty}\frac{1}{n}\log \mu^\NN \big( Y(n,\ep) \big)\,\,\,=\,\,\,\lim_{n\,\to\,+\infty}\frac{1}{n}\log \mu^\NN \big( Y_{n,\ep} \big).$$
\end{lemma}
\begin{proof}
    Fix $\ep>0$. For every $y=(y_1,y_2,...)\in Y$ and $n,k\in\NN,$ \begin{align*}
        B_\ep(y_1)\times\cdots\times B_\ep(y_{n+k}) \,\,&\subset \,\,B_\ep(y_1)\times\cdots\times B_\ep(y_{k})\,\,\,\times \bigcup_{z\,=(z_i)_i\,\in\,Y}B_\ep(z_{k+1})\times\cdots\times B_\ep(z_{n+k}) \\ & \subset \,\,B_\ep(y_1)\times\cdots\times B_\ep(y_{k})\,\,\,\times \bigcup_{w\,=(w_i)_i\,\in\,Y}B_\ep(w_{1})\times\cdots\times B_\ep(w_{n}) \\ & \subset \,\,B_\ep(y_1)\times\cdots\times B_\ep(y_{k})\,\,\,\times \,\, Y(n,\ep).
    \end{align*} Therefore, \begin{align*}
        \mu^{n+k} \bigg(Y(n+k,\ep) \bigg)\,\,&\leq \,\,\mu^{n+k}\bigg(\bigcup_{y\,=\,(y_i)_i\,\in\,Y}B_\ep(y_1)\times\cdots\times B_\ep(y_{k})\,\,\,\times \,\, Y(n,\ep)\bigg) \\ & = \,\,\mu^{n+k}\bigg(Y(k,\ep)\,\,\times \,\,Y(n,\ep)\bigg)\\ & = \,\,\mu^{k}\bigg(Y(k,\ep)\bigg)\,\, \cdot\,\,\mu^n\bigg(Y(n,\ep)\bigg).
    \end{align*}
\end{proof}

Fix $\ep>0$ and $n\geq 1$. If $Y$ can be covered by $N(n,\ep)$ distinct $(n,\ep)-$dynamical balls, then $Y_{n,\ep}$ can be covered by the concentric dynamical balls of radius $2\ep.$ Then 
$$\mu^\NN \big( Y_{n,\ep} \big) \, \leq \, N(n,\ep)  \big( C_2 (2\ep)^{\mathrm{dim}X}\big)^n,$$
and therefore
\begin{equation}\label{tuboupperbound}
    \lim_{n\,\to\,+\infty}\frac{1}{n}\log \mu^\NN \big( Y_{n,\ep} \big)\,\leq\, \limsup_{n\,\to\,+\infty}\frac{1}{n}\log N(n,\ep) + \log\big(C_2(2\ep)^{\mathrm{dim}X}\big).
\end{equation} 
Hence, 
$$ \limsup_{\ep\,\to\,0^+} \frac{1}{\log(1/\ep)} \lim_{n\,\to\,+\infty}\frac{1}{n}\log \mu^\NN \big( Y_{n,\ep} \big)\,\leq\,\overline{\mathrm{mdim}}_M(Y,d^\NN ,\si_\NN) \,-\,\dim X. $$

For the converse inequality, if there are $S_1(n,\ep)$ distinct $(n,\ep)-$separated points $y_1,\cdots,y_{S_1(n,\ep)}$ in $Y$, then the $(n,\ep/3)-$dynamical balls centered at these points are disjoint. Then 
$$\mu^\NN \big( Y_{n,\ep} \big) \, \geq \, \sum_{i=1}^{S_1(n,\ep)} \mu^\NN \big( B_{\ep/3}(y_i,d_n) \big) \,\geq\, S_1(n,\ep) \big( C_1(\ep/3)^{\dim X}\big)^{n+\ell(\ep)},$$ 
so  
\begin{equation}\label{tubolowerbound}
    \lim_{n\,\to\,+\infty}\frac{1}{n}\log \mu^\NN \big( Y_{n,\ep} \big)\,\geq\, \limsup_{n\,\to\,+\infty}\frac{1}{n}\log S_1(n,\ep) + \log(C_1(\ep/3)^{\dim X}).
\end{equation} 
Thus, 
$$ \limsup_{\ep\,\to\,0^+} \frac{1}{\log(1/\ep)} \lim_{n\,\to\,+\infty}\frac{1}{n}\log \mu^\NN \big( Y_{n,\ep} \big)\,\geq\,\overline{\mathrm{mdim}}_M(Y,d^\NN ,\si_\NN) \,-\,\dim X. $$  This concludes the proof of \eqref{eq1mdimsubshift}. The statement for the lower metric mean dimension is proved analogously.

Regarding the entropy, taking the limsup as $\ep\to0^+$ in \eqref{tuboupperbound} and \eqref{tubolowerbound}, we obtain \begin{equation}\label{h1h2}
    h_1\,\leq\,h_{top}(Y,\si_\NN)\,\leq\,h_2
\end{equation} where 
\begin{eqnarray*}    h_1&=&\limsup_{\ep\,\to\,0^+}\lim_{n\,\to\,+\infty}\frac{1}{n}\log \Big(\frac{\mu^n\big(Y(n,\ep)\big)}{(C_2\,2^{dim X})^n\,\ep^{n\,\dim X}}\Big) \\ h_2&=&\limsup_{\ep\,\to\,0^+}\lim_{n\,\to\,+\infty}\frac{1}{n}\log \Big(\frac{\mu^n\big(Y(n,\ep)\big)}{(C_1\,3^{-\dim X})^n\,\ep^{n\,\dim X}}\Big).
\end{eqnarray*}
When $\dim X=0$, $X$ is a finite set, say $X=\{x_1,...,x_q\}$. If we consider the normalized counting measure $\mu$, then $C_1=C_2=1/q$ and \eqref{entformula}, in this setting, is an immediate consequence of \eqref{h1h2}, with $\la=1/q$. On the other hand, if $\dim X>0$, we have two possibilities. If $h_{top}(Y,\si_\NN)=+\infty$, then \eqref{entformula} holds with $\la= C_1\,3^{-\dim X}$. If $h_{top}(Y,\si_\NN)<+\infty$, let us show that $$-\infty\,\,\,<\,\,\,\limsup_{\ep\,\to\,0^+}\lim_{n\,\to\,+\infty}\frac{1}{n}\log \Big(\frac{\mu^n\big(Y(n,\ep)\big)}{\ep^{n\,\dim X}}\Big)\,\,\,<\,\,\,+\infty. $$ Indeed, by Definition~\ref{dimhommms}, for any $y\in Y$ one has $$\mu^n\big(Y(n,\ep)\big)\,\,\,\geq\,\,\,\mu^n\big( B_\ep(y_1)\times...\times B_\ep(y_n) \big)\,\,\,\geq (C^1\,\ep^{\dim X})^n$$ and therefore $$\limsup_{\ep\,\to\,0^+}\lim_{n\,\to\,+\infty}\frac{1}{n}\log \Big(\frac{\mu^n\big(Y(n,\ep)\big)}{\ep^{n\,\dim X}}\Big)\,\,\,\geq\,\,\,\log C_1.$$ Moreover, $$\limsup_{\ep\,\to\,0^+}\lim_{n\,\to\,+\infty}\frac{1}{n}\log \Big(\frac{\mu^n\big(Y(n,\ep)\big)}{\ep^{n\,\dim X}}\Big)\,\,\,\leq\,\,\,h_{top}(Y,\si_\NN)\,\,\,+\,\,\, \log \big( C_2 \,2^{\dim X}\big)\,\,\,<\,\,\,+\infty. $$
Now, consider the strictly decreasing map $G \colon ]0, +\infty[\to\RR$  given by
$$G(t)\,\,=\,\,\,\limsup_{\ep\,\to\,0^+}\lim_{n\,\to\,+\infty}\frac{1}{n}\log \Big(\frac{\mu^n\big(Y(n,\ep)\big)}{t^n\,\ep^{n\,\dim X}}\Big)\,\,\,=\,\,\,\limsup_{\ep\,\to\,0^+}\lim_{n\,\to\,+\infty}\frac{1}{n}\log \Big(\frac{\mu^n\big(Y(n,\ep)\big)}{\ep^{n\,\dim X}}\Big)\,\,-\,\,\log t.$$
Observe that $G$ is continuous, $G(C_2\, 2^{\dim X}) = h_1$ and $G(C_1 \, (1/3)^{\dim X}) = h_2$. Thus, due to \eqref{h1h2} and the intermediate value theorem, there exists a unique $\la \in [C_1 \,(1/3)^{\dim X}, C_2\, 2^{\dim X}]$ such that $G(\la) = h_{top}(Y, \si_\NN)$, as claimed. The proof of Theorem~\ref{mdimsubshift} is complete.\qed

\medskip

\subsection{Proof of Corollary~\ref{iterados}}
    For a map $F\colon X\,\to\, X$, $n\in\NN$ and $\ep>0$, we denote $$Y(F,n,\ep)\,=\,\bigcup_{x\,\in\, X} B_\ep(x)\times B_\ep(F(x))\times\cdots\times B_\ep(F^{n-1}(x)).$$ Fix $k\geq1$ and observe that, for every $\ep>0$ and large enough $n\geq1$, we have 
        \begin{align*}
             Y(T,kn,\ep)\,& =\,\bigcup_{x\,\in\, X} B_\ep(x)\times B_\ep(T(x))\times\cdots\times B_\ep(T^{kn-1}(x)) \\ &\subset\, \bigcup_{x\,\in\, X} B_\ep(x)\times X\times...\times X \\  & \hspace{1.3cm} \times B_\ep(T^k(x))\times X\times...\times X \\ &\hspace{1.6cm} \vdots \\  & \hspace{1.3cm} \times B_\ep(T^{(n-1)k}(x))\times X\times...\times X \\ &:=\,\,\,\,\,A.
        \end{align*}
Then
$$\mu^{kn}\big(Y(T,kn,\ep) \big)\,\,\,\,\leq\,\,\, \mu^{kn}\big(A \big)\,\,\,\,=\,\,\,\, \mu^{n}\big(Y(T^k,n,\ep) \big),$$ and therefore 
\begin{align*}
\lim_{n\,\to\,+\infty}\frac{1}{n}\log \mu^n \big(Y(T^k,n,\ep) \big)\, \,\,\,&\geq\,\,\, \, k \,\,\lim_{n\,\to\,+\infty}\frac{1}{kn}\log \mu^{kn} \big(Y(T,kn,\ep) \big) \\& = \,\,\,\, k\,\, \lim_{n\,\to\,+\infty}\frac{1}{n}\log \mu^{n} \big(Y(T,n,\ep) \big).
\end{align*} By Corollary~\ref{lowdimformula}, we conclude that
\begin{align*}
    \overline{\mathrm{mdim}}_M(X,d ,T^k)\, \,\,\, & =\,\,\,\,  \dim X\,\,\,+\,\,\, \limsup_{\ep\,\to\,0^+} \frac{1}{\log(1/\ep)} \lim_{n\,\to\,+\infty}\frac{1}{n}\log \mu^n \big( Y(T^k,n,\ep) \big) \\&\geq\,\,\,\,  \dim X\,\,\,+\,\,\, k\,\,\limsup_{\ep\,\to\,0^+} \frac{1}{\log(1/\ep)} \lim_{n\,\to\,+\infty}\frac{1}{n}\log \mu^n \big( Y(T,n,\ep) \big)\\&=\,\,\,\, k\,\,\,\overline{\mathrm{mdim}}_M(X,d ,T)\, \,\,-\,\,\,(k-1)\,\dim X.
\end{align*}
The statement for the lower metric mean dimension is proved analogously. \qed
\medskip

\section{Proof of Theorem~\ref{holderineq}}

   Let $(X,d)$ be a compact metric space of finite Assouad dimension with doubling constant $K$ (see Subsection~\ref{fractaldimensions}), and $T\colon X\to X$ be $\alpha-$H\"older for some $\alpha\,\in\,(0,1)$. We will make use of the information in Section~\ref{specradiussection} to estimate the metric mean dimension of $T$, by relating its H\"older exponent, the transition matrix \eqref{gammaepgeneral} and the Assouad spectrum of parameter $\theta=\alpha$ (see Definition~\ref{spectrum}).

 Fix $\ep>0$. Let $x_1,...,x_{N}\,\in\,X$ be a maximal collection of $\ep-$separated points and consider the open cover $\cU(2\ep)\,=\,\{ B_{\ep}(x_1),...,B_{\ep}(x_N) \}$ of $X$, whose diameter is $2\ep$. Let us estimate the number of elements in this cover that can intersect a given ball.

\begin{lemma}
    For every $z\in X$ and $R>2\ep>0$, $$ \# \{i:B_{\ep}(x_i)\,\cap\,B_{R}(z) \ne\emptyset\}\,\leq\, K\,\sup_{x\,\in\,X}\,S\big( B_R(x),d,\ep \big),$$ where $K$ is the doubling constant of $(X,d).$
\end{lemma}
\begin{proof}
    Recall that the balls $\{B_{\ep/2}(x_i):i=1,...,N\}$ are pairwise disjoint. Reordering if necessary, assume that the elements of $\cU(2\ep)$ that intersect $B_{R}(z)$ are $B_{\ep}(x_1),...,B_{\ep}(x_q).$ Then $\bigcup_{i=1}^q B_{\ep/2}(x_i)\,\subset B_{R+2\ep}(z)$, where the union is disjoint. Hence 
        $$q \, \leq\,S\big( B_{R+2\ep}(z),d,\ep \big)\,\leq\,S\big( B_{2R}(z),d,\ep \big)\,\leq\,K\,\sup_{x\,\in\, X}\, S\big( B_R(x),d,\ep \big).$$
\end{proof}
Define the matrix $\Ga^{2\ep}$ associated to the cover $\cU(2\ep)$ as in \eqref{gammaepgeneral}. Since the map $T$ is $\alpha-$H\"older, for every $i=1,...,N$ there exists some $z_i\in X$ such that $T\big(B_{\ep}(x_i)\big)\,\subset\, B_{C_\alpha\,\ep^\alpha}(z_i)$. By the previous Lemma, $T\big(B_{\ep}(x_i)\big)$ can intersect at most  \begin{equation}\label{ooo}
    K\,\sup_{x\,\in\, X}\, S\big( B_{C_\alpha\,\ep^\alpha}(x),d,\ep \big)\,\leq\, K^{1+M}\,\sup_{x\,\in\, X}\, S\big( B_{\ep^\alpha}(x),d,\ep \big)
\end{equation} elements of $\cU(2\ep)$ for every $i=1,...,N$, where $M=\big\lceil \frac{\log2}{\log C_\alpha}\big\rceil$ and $C_\alpha$ is the H\"older constant of $T$. This means that \eqref{ooo} provides an upper bound for the number of $1$'s on each line of $\Ga^{2\ep}$. By Lemma~\ref{specradgeneral} and \eqref{eqcircle} we have $$h_{2\ep}(\Ga_\NN,d^\NN,\si_\NN)\,\leq\,\log r(\Ga^{2\ep})\,\leq\,(1+M) \log K\,+\,\sup_{x\,\in\, X}\, \log S\big( B_{\ep^\alpha}(x),d,\ep \big),$$ where $\Ga_\NN$ is the subshift of compact type induced by $\Ga\,=\,\graph(T)$. Combining the above estimate with \eqref{invlimit} we conclude that$$\overline{\mathrm{mdim}}_M(X,d,T)\,=\,\overline{\mathrm{mdim}}_M(\Ga_\NN,d^\NN,\si_\NN)\,\leq\,(1-\alpha)\, \dim_A^\alpha(X,d).$$
This ends the proof of the first part of Theorem~\ref{holderineq}. The second part is a straightforward consequence of the last inequality when $\dim_A(X,d)>0$ and the right inequality in \eqref{ineqdimensions}.

\qed

\section{Proof of Theorem~\ref{mdiminterval}}\label{proofholder}

We will split the statement of Theorem~\ref{mdiminterval} in two. Let's start by applying Corollary~\ref{lowdimformula} to compute the upper and lower metric mean dimensions of $T_{a,b}$ (defined in Subsection~\ref{app}) with respect to the Euclidean metric.

\begin{proposition}\label{mdiminterval1}
Let $T_{a,b}$ be the map satisfying the conditions (C1)-(C3) in Theorem~\ref{mdiminterval}. Then 
\begin{eqnarray*}
\overline{\mathrm{mdim}}_M([0,1],|\cdot|,T_{a,b})&=&\limsup_{k\,\to\,+\infty}\frac{\log b_k}{\log(1/\ep_k)}\\
     \underline{\mathrm{mdim}}_M([0,1],|\cdot|,T_{a,b})&=&\liminf_{k\,\to\,+\infty}\frac{\log b_k}{\log(1/\ep_k)}.
     \end{eqnarray*}
\end{proposition}
\begin{proof}    

We start recalling that, for every $s\geq1$ and $n\geq1$, we can partition $J_s$ in $b_s^n$ subintervals $J_s(i_1,\cdots,i_n)$, for $i_1,...,i_n\in\{1,\cdots,b_s\}$, satisfying (see Example~$10.3$ in \cite{CPV}) $$T_{a,b}^j(J_s(i_1,\cdots,i_n))\,=\, J_s(i_{j+1},\cdots,i_n)\qquad \forall j\in \{1,\cdots,n-1\}.$$

Given $\ep>0$, let $k=k(\ep)$ be the unique positive integer satisfying $\ep_{k+1}\,\leq\,\ep\,<\,\ep_k$. Consider a positive integer $M=M(\ep)\gg k$ such that $a_{M-1}>1-\ep.$ Then, for every $n\geq1,
$\begin{align*}
Y(n,\ep)\,&=\,\bigcup_{x\,\in\, [0,1]} B_\ep(x)\times B_\ep(T(x))\times\cdots\times B_\ep(T^{n-1}(x)) \\& =\, \bigcup_{s=1}^M \,\Bigg(\,\bigcup_{x \,\in\, J_s} B_\ep(x)\times B_\ep(T(x))\times\cdots\times B_\ep(T^{n-1}(x)) \,\Bigg)\\&  :=\, Y([0,a_k],n,\ep)\,\cup\, Y([a_k,a_M],n,\ep),
    \end{align*}
where \begin{eqnarray*}
    Y([0,a_k],n,\ep) &=& \bigcup_{s=1}^k \,\Bigg(\,\bigcup_{x \,\in\, J_s} B_\ep(x)\times B_\ep(T(x))\times\cdots\times B_\ep(T^{n-1}(x)) \,\Bigg) \\ Y([a_k,a_M],n,\ep) &=& \bigcup_{s=k+1}^M \,\Bigg(\,\bigcup_{x \,\in\, J_s} B_\ep(x)\times B_\ep(T(x))\times\cdots\times B_\ep(T^{n-1}(x)) \,\Bigg).
\end{eqnarray*}

Let us estimate the $n-$th Lebesgue measure of each of these sets separately. Consider a minimal $(n,\ep)-$spanning set $x_1,\cdots,x_N$ in $[0,a_k].$ Then $$Y([0,a_k],n,\ep)\,\subset\, \bigcup_{i=1}^N B_{2\ep}(x_i)\times B_{2\ep}(T(x_i))\times\cdots\times B_{2\ep}(T^{n-1}(x_i)),$$ and therefore \begin{align}
    \limsup_{n\,\to\,+\infty}\frac{1}{n}\log \Leb^n \big( Y([0,a_k],n,\ep) \big)\,&\leq\, h_\ep([0,a_k])\,+\, \log(4\ep) \nonumber \\&\leq\, \log b_k\,+\,\log\ep_k \,+\,\log4  \nonumber\\&=\, \log|J_k|+\log4. \label{ub1} 
\end{align} Besides, by invariance of each $J_s$, we can consider the coarse estimate $$Y([a_k,a_M],n,\ep)\,\subset\, \bigcup_{s=k+1}^M \big(B_\ep(J_s)\big)^n.$$ As $(|J_s|)_s$ is decreasing, we get
\begin{align}
    \Leb^n \big( Y([a_k,a_M],n,\ep)\big) \,&\leq\, \sum_{s=k+1}^M\big(|J_s|+2\ep\big)^n \nonumber \\ &\leq\, \big(M(\ep) - k(\ep)\big) \,\cdot\, \big(|J_k|+2\ep_k\big)^n \nonumber \\&=\, \big(M(\ep) - k(\ep)\big) \,\cdot\, |J_k|^n\,\cdot\, \big(1+\frac{2}{b_k}\big)^n,\nonumber
\end{align} so 
\begin{equation}\label{ub2}
    \limsup_{n\,\to\,+\infty}\frac{1}{n}\log \Leb^n \big( Y([a_k,a_M],n,\ep)\big)\,\leq\, \log|J_k|\,+\,\log\big(1+\frac{2}{b_k}\big)\,\leq\,\log|J_k|\,+\,\log4.
\end{equation}

\smallskip
\noindent Bringing together \eqref{ub1}, \eqref{ub2} and Lemma~\ref{subadditive} we obtain \begin{equation}\label{ub}
    \lim_{n\,\to\,+\infty}\frac{1}{n}\log \Leb^n \big( Y(n,\ep)\big)\,\leq\, \log|J_k|\,+\,\log4.
\end{equation} 

 On the other hand, it is easy to see that for every $n\geq1$,  $J_{k+1}^n\subset Y(n,\ep)$. Indeed, given $(x_1,\cdots,x_n) \in J_{k+1}^n$ there exists $i_j$ such that $x_j\,\in\,J_{k+1}(i_j)$ for every $j=1,...,n$.
Then, any element  $y\in J_{k+1}(i_1,\cdots,i_n)$ $\ep_{k+1}-$shadows this trajectory along the first $n$ iterates. Therefore, $$(x_1,\cdots,x_n)\,\in B_\ep(y)\times B_\ep(T(y))\times\cdots\times B_\ep(T^{n-1}(x))\,\subset\,Y(n,\ep).$$ Combining this information with \eqref{ub} we get \begin{equation*}
   \log|J_{k+1}|\,\leq\, \lim_{n\,\to\,+\infty}\frac{1}{n}\log \Leb^n \big( Y(n,\ep) \big)\,\leq\,\log|J_k|+\log4,
\end{equation*} thus \begin{equation*}
     \frac{\log|J_{k+1}|}{\log(1/\ep_{k+1})}\,\leq\, \frac{1}{\log(1/\ep)}\lim_{n\,\to\,+\infty}\frac{1}{n}\log \Leb^n \big( Y(n,\ep) \big)\,\leq\,\frac{\log|J_k|+\log4}{\log(1/\ep_k)}.
\end{equation*}

\medskip

\noindent Proposition~\ref{mdiminterval1} now follows from Corollary~\ref{lowdimformula} and the fact that $$1\,+\, \frac{\log|J_{s}|}{\log(1/\ep_{s})}\,=\, \frac{\log b_s}{\log(1/\ep_{s})} \qquad\forall s\geq1.$$

\end{proof}

The previous reasoning does not depend on the affine structure of the $f_k$'s. In fact, not even continuity was used: as long as, for every $k\geq1$, $\tilde{f}_k$ satisfies both the geometric constraint (equidistribution of critical points) and the topological constraint (strictly monotone full branches), the maps $\tilde{T}_{a,b}$ and $T_{a,b}$ have the same upper and lower metric mean dimensions. 

To end the proof of Theorem~\ref{mdiminterval} we are left to link the upper metric mean dimension and the H\"older regularity of the map $T_{a,b}$. 

\smallskip

\begin{proposition}\label{mdiminterval2}
    Let $T_{a,b}$ be the map satisfying the conditions (C1)-(C3) in Theorem~\ref{mdiminterval}. Then, $$\sup\cH(T_{a,b})\,=\,1\,-\,\overline{\mathrm{mdim}}_M([0,1],|\cdot|,T_{a,b}),$$ where $\cH(T_{a,b})\,=\,\{\alpha \in(0,1): T_{a,b} \,\,\,\text{is}\,\,\, \alpha -\text{H\"older}\}\,\cup\,\{0\}.$
\end{proposition}

\begin{proof}
    Let $T=T_{a,b}$. By Theorem~\ref{holderineq} we already know that $$\sup\cH(T)\,\leq\,1\,-\,\overline{\mathrm{mdim}}_M([0,1],|\cdot|,T).$$
    
    For the converse inequality, we may assume that $\overline{\mathrm{mdim}}_M([0,1],|\cdot|,T)<1$, otherwise there is nothing else to prove. Let $\alpha\,\in\,(0,\,1-\overline{\mathrm{mdim}}_M([0,1],|\cdot|,T))$. Our aim is to bound $$\sup_{x\,\neq \,y\,\in\, [0,1]}\frac{|T(x)-T(y)|}{|x-y|^\alpha}.$$ To this end we will consider the possible different location of the points $x,y$ with respect to the decomposition provided by $(J_k)_k.$

\medskip

\noindent\textbf{Case 1:} $x,y\in J_k$ for some $k\geq1.$

\medskip

Let
\begin{equation}\label{H_kalpha}
        H_k(\alpha)\,=\, \sup_{x\,\neq\, y\,\in \,J_k}\frac{|T(x)-T(y)|}{|x-y|^\alpha} \, =\,\frac{|T(a_{k-1})-T(a_{k-1}+\ep_k)|}{\ep_k^\alpha}\,=\,\frac{|J_k|}{\ep_k^\alpha}\,=\,\frac{b_k^\alpha}{|J_k|^{\alpha-1}}.
    \end{equation}
Observe that \begin{eqnarray}\label{logH_k}
    \log H_k(\alpha)&=& \alpha \,\log b_k\,-\,(\alpha\,-\,1)\,\log |J_k| \nonumber\\
    &=& \log b_k\, \Big( (\alpha\,-\,1)\big( 1\,-\,\frac{\log|J_k|}{\log b_k}\big)\,+1\Big) \nonumber\\ 
    &=& \log b_k\, \Big( (\alpha\,-\,1)\big( \frac{\log(1/\ep_k)}{\log b_k}\big)\,+1\Big).
\end{eqnarray}
    
\noindent The maximal distortion inside $J_k$ is given precisely by $H_k(\alpha)$. By choice of $\alpha$, we know that $$\limsup_{k\,\to\,+\infty}\Big( (\alpha-1)(\frac{\log(1/\ep_k)}{\log b_k}) +1 \Big)\,<\, 0.$$ In particular, combining this information with \eqref{logH_k} we obtain that $H_k(\alpha)$ is strictly smaller than $1$ for large enough $k$. Hence, $(H_k(\alpha))_k$ is uniformly bounded and we obtain \begin{equation}\label{case1}
    \sup_{k\,\geq\,1}\, \sup_{x\,\neq\, y\,\in\, J_k} \frac{|T(x)-T(y)|}{|x-y|^\alpha}\,=\,\sup_{k\,\geq\,1} H_k(\alpha)\,<\,+\infty.
\end{equation}

\medskip

\noindent\textbf{Case 2:} $x\in J_k$ for some $k\geq1$ and $y=1$.

\medskip

The maximal distortion between $1$ and some element in $x\in J_k$ is attained at $x= a_k-\ep_k$, namely $$\sup_{x\,\in\, J_k}\frac{|1-T(x)|}{|1-x|^\alpha}\,=\,\frac{1-a_{k-1}}{(1-a_k+\ep_k)^\alpha}.$$ By concavity of the map $t\mapsto t^\alpha$, we have \begin{align*}
    \frac{1-a_{k-1}}{(1-a_k+\ep_k)^\alpha} \,& = \,\frac{(1-a_{k})\,+\,|J_k|}{(1-a_k+\ep_k)^\alpha} \\ &\leq\, 2^{1-\alpha}\Big(\frac{(1-a_{k})\,+\,|J_k|}{(1-a_k)^\alpha +\ep_k^\alpha}\Big) \\ &\leq\, 2^{1-\alpha}\Big( (1-a_k)^{1-\alpha} + \frac{|J_k|}{\ep_k^\alpha} \Big) \\ &\leq\, 2^{1-\alpha}\Big( 1 + H_k(\alpha) \Big).
\end{align*}
Thus, \begin{equation}\label{case2}
    \sup_{k\,\geq\,1}\,\sup_{x\,\in\, J_k}\frac{|1-T(x)|}{|1-x|^\alpha}\,<\,+\infty.
\end{equation}

\medskip

\noindent\textbf{Case 3:} $x\in J_k$ and $y\in J_{k+n}$ for some $k,n\geq1$.

\medskip

The maximal distortion between $x\in J_k$ and $y\in J_{k+n}$ is attained at $x= a_k-\ep_k$ and $y=a_{k+n-1}+\ep_{k+n}$, namely $$\sup_{x\,\in\, J_k,\,y\,\in\, J_{k+n}}\frac{|T(y)-T(x)|}{|y-x|^\alpha}\,=\,\frac{a_{k+n}-a_{k-1}}{(a_{k+n-1}+\ep_{k+n} -a_k +\ep_k)^\alpha}\,=\, 
\frac{\sum_{i=k}^{k+n}|J_i|}{(\ep_k + \sum_{i=k+1}^{k+n-1}|J_i| + \ep_{k+n})^\alpha},$$ where the denominator is $(\ep_k\,+\,\ep_{k+n})^\alpha$ if $n=1.$ Again by concavity, we obtain
\begin{align*}
\frac{\sum_{i=k}^{k+n}|J_i|}{(\ep_k + \sum_{i=k+1}^{k+n-1}|J_i| + \ep_{k+n})^\alpha}\, & \leq\, 3^{1-\alpha}\Big(
\frac{|J_k|\,+\,\sum_{i=k+1}^{k+n-1}|J_i| \,+\,|J_{k+n}|}{\ep_k^\alpha + (\sum_{i=k+1}^{k+n-1}|J_i|)^\alpha + \ep_{k+n}^\alpha} \Big) \\ &\leq\,  3^{1-\alpha}\Big(
\frac{|J_k|}{\ep_k^\alpha}\,+\,(\sum_{i=k+1}^{k+n-1}|J_i|)^{1-\alpha} \,+\,\frac{|J_{k+n}|}{\ep_{k+n}^\alpha} \Big) \\ &\leq\,  3^{1-\alpha}\Big(
H_k(\alpha)\,+\,1 \,+\,H_{k+n}(\alpha) \Big).
\end{align*}
Hence, \begin{equation}\label{case3}
   \sup_{k,n\,\geq\,1}\, \sup_{x\,\in\, J_k,\,y\,\in\, J_{k+n}}\frac{|T(y)-T(x)|}{|y-x|^\alpha}\,<\,+\infty.
\end{equation}

\medskip

Bringing \eqref{case1}, \eqref{case2} and \eqref{case3} together we conclude that $$ \sup_{x\,\neq\, y\,\in\,[0,1]}\frac{|T(y)-T(x)|}{|y-x|^\alpha}\,<\,+\infty,$$ which means that $T$ is $\alpha-$H\"older as claimed.

\end{proof}

\section{Open questions}

In \cite{FHT} the authors showed that, given a compact manifold $X$ of dimension greater than one, the closure of the space of bi-Lipschitz homeomorphisms of $X$ with respect to the H\"older topology contains a residual subset of homeomorphisms with infinite topological entropy. Theorem~\ref{holderineq} suggests that we ask:

\begin{question}
    Given a compact smooth manifold $(X,d)$ with dimension $\dim X\,\geq\,2$ and  $\alpha\in(0,1)$, let $\cL_\alpha$ be the closure of the space of bi-Lipschitz homeomorphisms of $X$ with respect to the $\alpha-$H\"older norm $$||T||_\alpha\,=\,||T||_\infty\,+\,\sup_{x\neq y\,\in\,X}\frac{d(T(x),T(y))}{d(x,y)^\alpha}.$$
    Is it true that a generic map $T\in\cL_\alpha$ satisfies $$\overline{\mathrm{mdim}}_M(X,d,T)\,=\,(1-\alpha)\,\dim X?$$  
\end{question}

In Theorem~\ref{mdiminterval} we express the metric mean dimension of a family of interval maps in terms of the data of a sequence of horseshoes in a manner similar to Misiurewicz formula for the entropy. The main difference is the relevance of the scales at which the dynamics of each horseshoe can be detected. This is not surprising since the metric mean dimension, contrary to the entropy, is not purely topological, depending also on the geometry of the underlying space.

\begin{question}
    Given a continuous interval endomorphism with positive metric mean dimension, can we find a sequence of horseshoes for iterates of the map and a sequence of scales upon which we may build a formula for the metric mean dimension similar to the one of Misiurewicz?
\end{question}

For finite entropy maps acting on Ahlfors regular spaces, the computation of the entropy provided by Theorem~\ref{mdimsubshift} brings forth a constant $\la$, whose nature is worthwhile exploring.

\begin{question}
    Is there a dynamical or geometrical interpretation of this constant?
\end{question}

\section{Acknowledgments}
The authors are grateful to the anonymous referee for the valuable comments and to Paulo Varandas for the insightful discussions during the preparation of this paper. MC and GP were partially supported by CMUP, member of LASI, which is financed by national funds through FCT - Funda\c c\~ao para a Ci\^encia e a Tecnologia, I.P., under the projects with references UIDB/00144/2020 and UIDP/00144/2020. MC also acknowledges financial support from the project PTDC/MAT-PUR/4048/2021. GP has been awarded a PhD grant by FCT, with reference UI/BD/152212/2021.

\end{document}